\documentclass[preprint,12pt]{elsarticle}

\usepackage{graphicx}
\usepackage{amssymb}
\usepackage{amsthm}

\usepackage{lineno}

\usepackage{caption}
\usepackage{amsmath,algorithm,algorithmic}
\newcommand{\pdv}[3][]{\frac{\partial^{#1}#2}{\partial#3^{#1}}}
\newcommand{\dv}[3][]{\frac{\text{d}^{#1}#2}{\text{d}#3^{#1}}}
\newcommand{\R}{\mathbb{R}}

\newcommand{\V}{\mathbf{V}}
\newcommand{\U}{\mathbf{U}}
\newcommand{\A}{\mathbf{A}}
\newcommand{\B}{\mathbf{B}}

\newcommand{\s}{\mathbf{s}}
\renewcommand{\S}{\mathbf{S}}
\newcommand{\f}{\mathbf{f}}

\newcommand{\w}{\mathbf{w}}

\newcommand{\W}{\mathbf{W}}
\renewcommand{\H}{\mathbf{H}}

\newcommand{\Ur}{\mathbf{U}_{r}}
\renewcommand{\Pr}{\mathbf{P}_{r}}
\newcommand{\Phir}{\mathbf{\Phi}_{r}}

\newcommand{\cT}{\mathcal{T}}
\newcommand{\bT}{\mathbf{T}}

\newcommand{\vs}{s}
\newcommand{\vf}{f}
\newcommand{\vw}{w}
\newcommand{\va}{a}
\newcommand{\vh}{h}

\newtheorem{lemma}{Lemma}
\newtheorem{theorem}{Theorem}
\newtheorem{definition}{Definition}
\newtheorem{example}{Example}

\usepackage{cleveref}
\usepackage{color}

\journal{Physica D}

\begin{document}

\begin{frontmatter}


\title{Lift \& Learn: \\ Physics-informed machine learning for \\large-scale nonlinear dynamical systems}



\author[eq]{Elizabeth Qian\corref{cor1}}
\author[bk]{Boris Kramer}
\author[bp]{Benjamin Peherstorfer}
\author[kw]{Karen Willcox}

\address[eq]{Center for Computational Engineering, Massachusetts Institute of Technology}
\address[bk]{Department of Mechanical and Aerospace Engineering, University of California San Diego}
\address[bp]{Courant Institute of Mathematical Sciences, New York University}
\address[kw]{Oden Institute for Computational Engineering and Sciences, University of Texas at Austin}

\begin{abstract}
	We present \emph{Lift \& Learn}, a physics-informed method for learning low-dimensional models for large-scale dynamical systems. The method exploits knowledge of a system's governing equations to identify a coordinate transformation in which the system dynamics have quadratic structure. This transformation is called a lifting map because it often adds auxiliary variables to the system state. The lifting map is applied to data obtained by evaluating a model for the original nonlinear system. This lifted data is projected onto its leading principal components, and low-dimensional linear and quadratic matrix operators are fit to the lifted reduced data using a least-squares operator inference procedure. 
	Analysis of our method shows that the Lift \& Learn models are able to capture the system physics in the lifted coordinates at least as accurately as traditional intrusive model reduction approaches. This preservation of system physics makes the Lift \& Learn models robust to changes in inputs. Numerical experiments on the FitzHugh-Nagumo neuron activation model and the compressible Euler equations demonstrate the generalizability of our model.
\end{abstract}


\end{frontmatter}

\section{Introduction}
The derivation of low-dimensional models for high-dimensional dynamical systems from data is an important task that makes feasible many-query computational analyses like uncertainty propagation, optimization, and control.
Traditional model reduction methods rely on full knowledge of the system governing equations as well as the ability to intrusively manipulate solver codes. In contrast, classical machine learning methods fit models to data while treating the solver as a black box, ignoring knowledge of the problem physics. In this paper, we propose a new hybrid machine learning-model reduction method called Lift \& Learn, in which knowledge of the system governing equations is exploited to identify a set of lifted coordinates in which a low-dimensional model can be learned.

In projection-based model reduction, the system governing equations are  projected onto a low-dimensional approximation subspace to obtain a reduced model. The basis for the reduced space is computed from data; a common choice is the Proper Orthogonal Decomposition (POD) basis, comprised of the leading principal components of the data~\cite{lumley1967structure,sirovich87turbulence,berkooz1993proper,holmes2012}. However, the projection process is intrusive, requiring access to the codes that implement the high-dimensional operators of the original equations. One way to avoid the intrusive projection is to learn a map from input parameters to coefficients of the reduced basis, e.g., via gappy POD~\cite{bui2004aerodynamic,mifsud2019fusing}, or using radial basis functions~\cite{Nair2013}, a neural network~\cite{wang2019non,mainini2017data,swischuk2018physics} or a nearest-neighbors method~\cite{swischuk2018physics}. However, these approaches are agnostic to the dynamics of the system in that they do not model the evolution of the system state over time.

To learn models for the dynamics of a system, sparsity-promoting regression techniques are used in~\cite{brunton2016discovering,rudy2017data} to fit terms of governing equations from a dictionary of possible nonlinear terms. The work in~\cite{schaeffer2018extracting} also uses sparse regression to learn sparse, high-dimensional matrix operators from data. However, these approaches do not reduce the dimension of the system. To learn a reduced model, dynamic mode decomposition (DMD)~\cite{rowley2009spectral,schmid2010dynamic} projects data onto a reduced basis and then fits to the reduced data a linear operator. Similarly, the operator inference approach of~\cite{Peherstorfer16DataDriven} fits linear and polynomial matrix operators to reduced data. However, if the true underlying dynamics of the system are non-polynomial, then the DMD and operator inference models may be inaccurate.

Several communities have explored using variable transformations to expose structure in a nonlinear system. In~\cite{bouvrie2017kernel}, a nonlinear system is embedded in a reproducing kernel Hilbert space in order to use linear balanced truncation for control. Koopman operator theory, which states that every nonlinear dynamical system can be exactly described by an infinite-dimensional linear operator that acts on scalar observables of the system~\cite{koopman1932dynamical}, has been used to extend DMD to fit linear models for nonlinear systems in observable space defined by a dictionary~\cite{williams2015data}, kernel~\cite{kevrekidis2016kernel}, or dictionary learning method~\cite{li2017extended}. In contrast, \emph{lifting} transformations~\cite{mccormick1976computability} yield a finite-dimensional coordinate representation in which the system dynamics are quadratic~\cite{gu2011qlmor}. 
Unlike the Koopman operator, the existence of a finite-dimensional quadratic representation of the system dynamics is not universally guaranteed. However,  a large class of nonlinear terms that appear in PDEs in engineering applications can be lifted to quadratic form via the process in~\cite{gu2011qlmor}. Currently, the lifting must be specifically derived for each new set of equations, but once the quadratic system is obtained it admits an explicit model parametrization that we will exploit for learning, and it can be more easily analyzed than the original nonlinear system.
Lifting is exploited for model reduction in~\cite{gu2011qlmor,kramer2018lifting,benner2015two,bennergoyal2016QBIRKA}, where the quadratic operators of a high-dimensional lifted model are projected onto a reduced space to obtain a quadratic reduced model. However, in many settings, it is impossible or impractical to explicitly derive a high-dimensional lifted model, so the only available model is the original nonlinear one.

In \emph{Lift \& Learn}, we use the available nonlinear model to learn quadratic reduced model operators for the lifted system without requiring a high-dimensional lifted model to be available. We simulate state trajectory data by evaluating the original nonlinear model, lift the data, project the lifted data onto a low-dimensional basis, and fit reduced quadratic operators to the data using the operator inference procedure of~\cite{Peherstorfer16DataDriven}. In this way, we learn a quadratic model that respects the physics of the original nonlinear system. Our contributions are thus:
\begin{enumerate}
	\item the use of lifting to explicitly parametrize the reduced model in a form that can be learned using the operator inference approach of~\cite{Peherstorfer16DataDriven},
	\item the use of learning to non-intrusively obtain lifted reduced models from data generated by the original nonlinear model, enabling their use even in settings where lifted models are unavailable, and
	\item the exploitation of the preservation of system physics to prove guarantees about the lifted model fit to data.
\end{enumerate}
Our approach fits reduced quadratic operators to the data in state space. In the situation where frequency-domain data is available, the work in~\cite{gosea2018LoewnerQB} fits quadratic operators to data in frequency space. 
\Cref{sec: background} describes projection-based model reduction for nonlinear systems. \Cref{sec: lift and learn} defines the lifting map and its properties and introduces our Lift \& Learn model learning method. \Cref{sec: theory} derives a bound on the mismatch between the data and the Lift \& Learn model dynamics.
In \Cref{sec: results} we apply our method to two examples: the FitzHugh-Nagumo prototypical neuron activation model and the Euler fluid dynamics equations. Our numerical results demonstrate the learned models' reliability and ability to generalize outside their training sets. 

\section{Projection-based model reduction}\label{sec: background}
Let $\Omega\in\R^d$ denote a physical domain and let $[0,T_{\text{final}}]$ be a time domain for some final time $T_{\text{final}}>0$. The nonlinear partial differential equation (PDE)
\begin{align}
	\pdv{\vs }t = \vf\left(\vs\right)
	\label{eq: nonlinear PDE}
\end{align}
defines a dynamical system for the $d_s$-dimensional vector state field
\begin{align}
	\vs(x,t) = \begin{pmatrix}
		s_1(x,t) \\ \vdots \\ s_{d_s}(x,t)
	\end{pmatrix},	
\end{align} 
where $s_j : \Omega \times [0,T_{\text{final}})\rightarrow \mathcal{S}_j\subset\R$, for $j =1,2,\ldots, d_s$, and where
\begin{align}
	\vf(\vs) = \begin{pmatrix}
		f_1(\vs) \\ \vdots \\ f_{d_s}(\vs)
	\end{pmatrix}
	\label{eq: vec f def}
\end{align}
is a differentiable nonlinear function that maps the state field to its time derivative. We assume that $\vs$ is sufficiently smooth so that any spatial derivatives of $\vs$ appearing in $\vf$ are well-defined. Denote by $\mathcal{S}$ the $d_s$-dimensional product space $\mathcal{S} = \mathcal{S}_1\times\cdots\times\mathcal{S}_{d_s}$, so that $\vs(x,t)\in\mathcal{S}$.

We consider a semi-discrete model where the spatially discretized state vector $\s(t)\in\R^{nd_s}$ corresponds to the values of the $d_s$ state variables at some collection of spatial points $\{x_l\in\Omega\}_{l=1}^n$, e.g., in a finite difference discretization or a finite element setting with a nodal basis. Let $\mathcal{S}^n$ denote the product domain $\mathcal{S}^n = \Pi_{l=1}^n \mathcal{S}$.
Then, the semi-discrete full model is given by a system of $nd_s$ ordinary differential equations:
\begin{align}
	\dv{\s}t = \f(\s),
  \label{eq: full model}
\end{align}
where $\f:\mathcal{S}^n\to\mathcal{S}^n$ is a function that discretizes $\vf$. The discrete nonlinear operator $\f$ is assumed to be Lipschitz continuous so that \cref{eq: full model} admits a unique solution. Projection-based model reduction seeks a low-dimensional approximation to \cref{eq: full model} to achieve computational speed-ups.

Denote by $\s_k$ the state snapshot at time $t_k$, i.e., the solution of \cref{eq: full model} at time $t_k$. The state snapshot matrix, $\S$, collects $K$ snapshots as
\begin{align}
	\S = \begin{bmatrix}
		\s_1 & \cdots & \s_K
	\end{bmatrix} \in \R^{nd_s\times K}.
	\label{eq: original full model data}
\end{align}
Note that $\S$ can contain states from multiple full model evaluations, e.g., from different initial conditions. 
The singular value decomposition (SVD) of $\S$ is given by
\begin{align}
	\S = \mathbf{\Phi}\mathbf{\Xi}\mathbf{\Psi}^\top
\end{align}
where $\mathbf{\Phi}\in\R^{nd_s\times nd_s}$, $\mathbf{\Xi}\in\R^{nd_s\times nd_s}$, and $\mathbf{\Psi}\in\R^{nd_s\times K}$, and we assume that the singular values $\xi_1\geq\xi_2\geq\cdots$ in $\mathbf{\Xi}$ are in descending order. The Proper Orthogonal Decomposition (POD) basis of size $r$ is denoted by $\Phir$ and defined by the leading $r$ columns of $\mathbf{\Phi}$. The POD state approximation is $\s \approx \Phir\hat\s$, where $\hat\s\in\R^r$ is the reduced state. The POD reduced model is defined by Galerkin projection: 
\begin{align}
	\dv{\hat\s}t = \Phir^\top\f\left(\Phir\hat\s\right).
	\label{eq: POD reduced model}
\end{align}
When $\f$ has no particular structure, solving \cref{eq: POD reduced model} requires evaluating the $nd_s$-dimensional map $\f$, which is expensive. 
When the full model contains only polynomial nonlinearities in the state, however, the POD-Galerkin model preserves this structure and can be efficiently evaluated without recourse to high-dimensional quantities. That is, let
\begin{align}
	\pdv{\vs}t = \va(\vs) + \vh(\vs)
	\label{eq: linear quadratic PDE}
\end{align}
where $\va(\cdot)$ and $\vh(\cdot)$ are linear and quadratic operators, respectively, and let 
\begin{align}
	\dv{\s}t = \A\s + \H(\s\otimes\s)
	\label{eq: quadratic full model}
\end{align}
be a discretization of~\cref{eq: linear quadratic PDE},
where $\A\in\R^{nd_s\times nd_s}$ is a linear matrix operator, $\H\in\R^{nd_s\times n^2d_s^2}$ is a matricized tensor operator, and $\otimes$ denotes a column-wise Kronecker product. That is, for a column vector $\mathbf{b} = [b_1, b_2, \ldots, b_m]^\top\in\R^m$, $\mathbf{b}\otimes\mathbf{b}$ is given by
\begin{align}
	\mathbf{b} \otimes \mathbf{b} = \begin{bmatrix}
		b_1^2 & b_1b_2 & \cdots & b_1b_m & b_2b_1 & b_2^2 & \cdots b_2b_m &\cdots & \cdots & b_m^2
	\end{bmatrix}^\top\in \R^{m^2},
	\label{eq: column definition of khatri-rao}
\end{align}
and for a matrix $\mathbf{B} = [\mathbf{b}_1, \mathbf{b}_2,\ldots \mathbf{b}_K]\in\R^{m\times K}$, $\mathbf{B}\otimes\mathbf{B}$ is given by
\begin{align}
	\mathbf{B}\otimes\mathbf{B} = \begin{bmatrix}
		\mathbf{b}_1\otimes\mathbf{b}_1 & \mathbf{b}_2\otimes\mathbf{b}_2 & \dots & \mathbf{b}_K\otimes\mathbf{b}_K
	\end{bmatrix}\in\R^{m^2\times K}.
	\label{eq: matrix definition of khatri-rao}
\end{align}
The POD-Galerkin reduced model is then given by:
\begin{align}
	\dv{\hat\s}t = \hat\A\hat\s + \hat\H(\hat\s\otimes\hat\s),
	\label{eq: lifted POD reduced model}
\end{align}
where 
\begin{align}
	\hat\A = \Phir^\top\A\Phir, \quad \hat\H = \Phir^\top\H(\Phir\otimes\Phir)
	\label{eq: intrusive reduced matrix operators}
\end{align}
are reduced matrix operators. The cost of evaluating \cref{eq: lifted POD reduced model} depends only on the reduced dimension $r$. However, in many cases, including in our setting, the high-dimensional operators $\A$ and $\H$ are not available, so the reduced matrix operators cannot be computed as in~\cref{eq: intrusive reduced matrix operators} and must be obtained through other means.

\section{Lift \& Learn: Physics-informed learning for nonlinear PDEs}\label{sec: lift and learn}
Lift \& Learn is a method for learning quadratic reduced models for dynamical systems governed by nonlinear PDEs. The method exposes structure in nonlinear PDEs by identifying a lifting transformation in which the PDE admits a quadratic representation. Non-quadratic state data is obtained by evaluating the original nonlinear model (\cref{eq: full model}) and the lifting transformation is applied to this data. Quadratic reduced model operators are then fit to the transformed data. The result is an efficiently evaluable quadratic reduced model for the original nonlinear PDE.

\Cref{sec: lifting} introduces lifting transformations and describes their properties.
\Cref{ssec: data acquisition} describes how lifted reduced state data are obtained from a full model simulation in the original non-polynomial system variables. \Cref{sec: opinf} introduces the operator inference procedure of~\cite{Peherstorfer16DataDriven} used to learn the reduced model from the lifted data. \Cref{ssec: LL summary} summarizes the Lift \& Learn method.

\subsection{Exposing structure via lifting transformations}\label{sec: lifting}
We begin by introducing two new definitions for the lifting map and its reverse. These definitions are part of our contribution and form the basis of our theoretical analysis in Section~\ref{sec: theory}.
\begin{definition}
	Define the lifting map,
	\begin{align}
 		\cT: \mathcal{S} \to \mathcal{W} \subset \mathbb{R}^{d_w}, \quad d_w\geq d_s,
 		\label{eq: lifting}
	\end{align}
	and let $\vw(x,t) = \cT(\vs(x,t))$. $\cT$ is a \emph{quadratic lifting} of \cref{eq: nonlinear PDE} if the following conditions are met:
	\begin{enumerate}
	 	\item  the map $\cT$ is differentiable with respect to $\vs$ with bounded derivative, i.e., if $\mathcal{J}(\vs)$ is the Jacobian of $\cT$ with respect to $\vs$, then 
	 	\begin{align}
	 		\sup_{\vs\in\mathcal{S}}\|\mathcal{J}(\vs)\| \leq c,
	 	\end{align}
	 	for some $c>0$, and
	 	\item the lifted state $\vw$ satisfies
		\begin{align}
			\pdv{\vw}t= \va(\vw) + \vh(\vw),
			\label{eq: lifted PDE}			
		\end{align}
		where 
			\begin{align}
				\va(\vw) = \begin{pmatrix}
					a_1(\vw) \\ \vdots \\ a_{d_w}(\vw)
				\end{pmatrix},
				\qquad
				\vh(\vw) = \begin{pmatrix}
					h_1(\vw) \\ \vdots \\ h_{d_w}(\vw)
				\end{pmatrix},
			\end{align}
		for some linear functions $a_j$ and quadratic functions $h_j$, $j=1,2,\ldots, d_w$.
	\end{enumerate} 
	The $d_w$-dimensional vector field $\vw(x,t)$ is called the \emph{lifted state} and \cref{eq: lifted PDE}
	is the lifted PDE. 
	\label{def: lifting map}
\end{definition}
Note that this definition may be extended to PDEs that contain constant and input-dependent terms --- see \Cref{sec: FHN} for an example. Lifting maps that transform non-polynomial dynamics into higher-order polynomial dynamics are also possible, but we focus on the quadratic case in this work. We now define a reverse lifting map which embeds the lifted state in the original state space.

\begin{definition}
	Given a lifting map $\cT$, a map
	\begin{align}
		\cT^\dagger: \mathcal{W} \to \mathcal{S}
	 	\label{eq: reverse lifting}
	\end{align}
	is called a reverse lifting map if it is differentiable with respect to $\vw$ with bounded derivative on $\mathcal{W}$ and satisfies $\cT^\dagger(\cT(\vs)) = \vs$ for all $\vs\in\mathcal{S}$.
	\label{def: reverse lifting}
\end{definition}

We now present a simple illustrative example of lifting.

\begin{example}
	Consider the nonlinear PDE,
	\begin{align}
		\pdv st = -e^s.
		\label{eq: toy nonlinear PDE}
	\end{align}
	To lift \cref{eq: toy nonlinear PDE}, we define the auxiliary variable $-e^s$. That is, the lifting map and its reverse map are given by
\begin{align}
	\cT: s \mapsto \begin{pmatrix}
		s \\ -e^s
	\end{pmatrix}\equiv \begin{pmatrix}
		w_1 \\ w_2
	\end{pmatrix}=\vw, \qquad \cT^\dagger: \vw \mapsto w_1,
\end{align}so that the lifted system is quadratic:
\begin{align}
 	\pdv{}t \begin{pmatrix}
 		w_1 \\ w_2
 	\end{pmatrix} = \begin{pmatrix}
 		1 \\ -e^s
 	\end{pmatrix}\pdv st = \begin{pmatrix}
 		1 \\ -e^s
 	\end{pmatrix}(-e^s) = 
 	\begin{pmatrix}
 		w_2 \\ (w_2)^2
 	\end{pmatrix}.
\label{ex: toy lifting}
\end{align} 
\end{example}

The lifting map must be specifically derived for the problem at hand. One strategy for doing so is to introduce auxiliary variables for the non-quadratic terms in the governing PDE and augment the system with evolution equations for these auxiliary variables~\cite{gu2011qlmor}. The work in~\cite{gu2011qlmor} shows that a large class of nonlinear terms which appear in engineering systems may be lifted to quadratic form in this way with $d_w$ of the same order of magnitude as $d_s$, including monomial, sinusoidal, and exponential terms.
In~\cite{kramer2018lifting}, this strategy is used to lift PDEs that govern systems in neuron modeling and combustion to quadratic form.
Note that $\cT$ is generally non-unique (consider $w_2 = e^s$ in \Cref{ex: toy lifting}), and for a given $\cT$, the reverse lifting map is also non-unique. 
Because $\cT$ is non-invertible in general, the choice of $\cT^\dagger$ can be interpreted as a regularization choice that allows the lifted state to be continuously embedded in the original state space. Any choices of $\cT$ and $\cT^\dagger$ satisfying \Cref{def: lifting map,def: reverse lifting} can be used in our Lift \& Learn approach. However, each additional lifted variable introduced will increase the dimension of the learning problem, so liftings that introduce fewer auxiliary variables are preferred. The questions of whether there is a minimal number of auxiliary variables needed to obtain a quadratic system and how to derive this minimal lifting remain an open problem.

\subsection{Lifted training data}\label{ssec: data acquisition}
This section presents a method for obtaining data in the lifted variables from the available non-lifted model (\cref{eq: full model}).

\paragraph{State data}
We first collect original snapshot data by simulating the original full model~\cref{eq: full model}. Then, for each column of the data matrix (\cref{eq: original full model data}), we apply the lifting map node-wise to the discrete state to obtain lifted snapshot data. That is, the lifted data matrix $\W\in\R^{nd_w\times K}$ is given by
\begin{align}
	\W = \begin{bmatrix}
		\w_{1} & \cdots & \w_{K}
	\end{bmatrix} = 
	\begin{bmatrix}
		\bT(\s_1) & \cdots & \bT(\s_K)
	\end{bmatrix},
	\label{eq: lifted data matrix}
\end{align}
where $\bT$ denotes the discrete lifting map defined by applying $\cT$ node-wise to each spatial node.

We denote the SVD of the transformed data by $\W = \U\mathbf{\Sigma}\V^\top$, with $\U\in\R^{nd_w\times nd_w}$, $\mathbf{\Sigma}\in\R^{nd_w\times nd_w}$, $\V\in\R^{nd_w\times K}$, where the singular values $\sigma_1\geq\sigma_2\geq\cdots$ in $\mathbf{\Sigma}$ are in descending order. The $r$-dimensional POD basis matrix is given by the leading $r$ columns of $\U$, denoted $\Ur\in\R^{nd_w\times r}$.
Projection onto $\Ur$ yields state data in the reduced space: 
\begin{align}
  \hat\W = \Ur^\top\W = \begin{bmatrix}
  	\hat\w_1 & \cdots & \hat\w_K
  \end{bmatrix}\in\R^{r\times K}.
  \label{eq: reduced transformed state data}
\end{align}

\paragraph{Time derivative data}
To learn the matrix operators of a lifted reduced model of the form in \cref{eq: lifted POD reduced model}, reduced state time derivative data are also required. To obtain data for dynamics that are Markovian in the reduced state, we adapt the procedure in~\cite{P19ReProj} to the Lift \& Learn setting.
For each $\hat\w_k$, 
\begin{align}
	\w_{\text{proj},k} = \Ur\hat\w_k = \Ur\Ur^\top \w_k\in\R^{nd_w}
\end{align} 
denotes the projection of the lifted state onto the subspace spanned by the POD basis. Denote by $\bT^\dagger$ the discrete reverse lifting defined by applying $\cT^\dagger$ node-wise for each spatial node. We assume that $\Ur$ is a sufficiently rich basis that $\bT^\dagger$ is well-defined for $\w_{\text{proj},k}$, and reverse the lifting for all $k$ to obtain a discrete non-lifted state that corresponds to the projected lifted state:
\begin{align}
 	\s_{\text{proj},k} = \bT^\dagger(\w_{\text{proj},k}).
\end{align} 
We then evaluate the available nonlinear full model to obtain new time derivative data, which we denote $\s_{\text{proj},k}'$:
\begin{align}
	\s_{\text{proj},k}'=\f\left(\bT^\dagger\left(\w_{\text{proj},k}\right)\right).
	\label{eq: sproj velocity}
\end{align}
Let $\mathbf{J}(\s)\in\R^{nd_w\times nd_s}$ denote the Jacobian of $\bT$ with respect to $\s$.
Applying the chain rule to~\cref{eq: sproj velocity} yields a time derivative in the discrete lifted state:
\begin{align}
	\w_{\text{proj},k}' = \mathbf{J}(\s_{\text{proj},k})\s_{\text{proj},k}'.
	\label{eq: lifted HD reprojected velocity}
\end{align}
Time derivative data in the reduced space is then obtained by projecting \cref{eq: lifted HD reprojected velocity} back onto the POD basis:
\begin{align}
	\hat\w'_k = \Ur^\top\mathbf{J}(\s_{\text{proj},k})\s_{\text{proj},k}' = \Ur^\top\left(\mathbf{J}(\bT^\dagger(\Ur\Ur^\top\w_k))\f\left(\bT^\dagger(\Ur\Ur^\top\w_k)\right)\right).
	\label{eq: lifted reprojected velocity reduced}
\end{align}
The reduced time derivative data matrix collects this data:
\begin{align}
	\hat\W' = \begin{bmatrix}
		\hat\w_1' & \cdots & \hat\w_K'
	\end{bmatrix}.
	\label{eq: reduced time derivative data}
\end{align}

\subsection{Least-squares operator inference procedure}\label{sec: opinf}

Given $K$ reduced state snapshots (\cref{eq: reduced transformed state data}) and the corresponding reduced time derivative data (\cref{eq: reduced time derivative data}) and a postulated model form (\cref{eq: lifted POD reduced model}), operator inference~\cite{Peherstorfer16DataDriven} formulates the following minimization problem for learning the matrix operators of \cref{eq: lifted POD reduced model}:
\begin{align}
  \min_{\hat\A\in\R^{r\times r},\hat\H\in\R^{r\times r^2}} \frac1K \left\| \hat\W^{\top}\hat\A^\top + \left(\hat\W\otimes\hat\W\right)^\top\hat\H^\top - {\hat\W}'^\top \right\|_F^2
  \label{eq: transposed minimization},
\end{align}
where $\otimes$ is the column-wise Kronecker product as defined in \cref{eq: matrix definition of khatri-rao}. Note that the square of the Frobenius norm of a matrix is equivalent to the sum of 2-norms of its rows. Thus, \cref{eq: transposed minimization} corresponds to an empirical risk minimization of the sum of square losses for each snapshot from $k=1,\ldots,K$.
Each column of $\hat\W^\top$, ${\hat\W}'^\top$, and $(\hat\W\otimes\hat\W)^\top$ corresponds to a single component of the reduced state $\hat\w$, yielding $r$ independent least-squares problems which each define one row of $\hat\A$ and $\hat\H$. Each least-squares problem has $r+\frac{r(r+1)}2$ degrees of freedom when the structural redundancy of the Kronecker product is accounted for. 

In practice, to solve \cref{eq: transposed minimization}, we form a data matrix,
\begin{align}
	\hat{\mathbf{D}} = \begin{pmatrix}
		\hat\W^\top & \hat\W_\text{sq}^\top
	\end{pmatrix}\in \R^{K \times \left(r+\frac{r(r+1)}2\right)},
\end{align}
where $\hat\W_\text{sq}\in \R^{\frac{r(r+1)}2\times r}$ contains quadratic data with the redundant cross terms of the Kronecker product removed. We then solve the least-squares equation,
\begin{align}
	\hat{\mathbf{D}}\begin{pmatrix}
		\hat\A^\top \\ \hat\H_\text{sq}^\top
	\end{pmatrix} = \hat\W',
	\label{eq: block LS problem}
\end{align}
where $\hat\H_\text{sq}\in\R^{r\times \frac{r(r+1)}2}$ contains coefficients of quadratic terms without redundancy, and we reconstruct the symmetric tensor $\hat\H\in\R^{r\times r^2}$ by splitting the coefficients for quadratic cross terms in $\hat\H_\text{sq}$ across the redundant terms. As long as $\hat{\mathbf{D}}$ has full column rank, \cref{eq: block LS problem} has a unique solution~\cite{golub1996matrix}. The work in~\cite{Peherstorfer16DataDriven} offers a data collection strategy for avoiding ill-conditioning in $\hat{\mathbf{D}}$. Normalizing the columns of $\hat{\mathbf{D}}$ can also improve its conditioning. In the numerical experiments of \Cref{sec: results}, a QR solver is used to solve \cref{eq: block LS problem}. 

The Lift \& Learn reduced model is then given by
\begin{align}
	\dv{\hat \w}t = \hat\A\hat\w + \hat\H (\hat\w\otimes\hat\w).
\end{align}
Note that the linear-quadratic case is considered for simplicity and concreteness, but the operator inference procedure can be flexibly adapted to learn reduced models with constant, input-dependent, and higher-order polynomial terms. \Cref{sec: FHN} contains an example in which input operators are learned.

\subsection{Lift \& Learn: Method summary}\label{ssec: LL summary}
\begin{algorithm}[t]
  \caption{Lift \& Learn}
  \label{alg:transform-learn}
  \begin{algorithmic}[1]
  \STATE{Use knowledge of governing PDE to identify lifting map $\cT$ as in \Cref{sec: lifting}.}
  \STATE{Evaluate non-quadratic full model (\cref{eq: full model}) to obtain state data in the original, non-lifted variables, $\S$.}
  \STATE{Use $\bT$ (defined by $\cT$) to transform nonlinear state data $\S$ to lifted data $\W$.}
  \STATE{Compute POD basis $\Ur$ for the lifted state data.}
  \STATE{Project to obtain reduced lifted state data (\cref{eq: reduced transformed state data}) and reduced lifted time derivative data (\cref{eq: reduced time derivative data}) as in \Cref{ssec: data acquisition}. }
  \STATE{Solve least-squares minimization (\cref{eq: transposed minimization}) using lifted data to infer operators $\hat\A$ and $\hat\H$.}
  \end{algorithmic}
\end{algorithm}

The Lift \& Learn method is summarized in \Cref{alg:transform-learn}. The first step is to identify an appropriate lifting transformation as described in \Cref{sec: lifting}. The available full model in the original non-lifted variables is then used to obtain lifted reduced state data as described in \Cref{ssec: data acquisition}. The operator inference framework in \Cref{sec: opinf} is then employed to learn a quadratic reduced model. Note that although this paper has primarily considered the case where the governing physics are described by nonlinear PDEs, the approach applies equally to systems where the governing equations are high-dimensional nonlinear ODEs.

\section{Finite difference analysis}\label{sec: theory}
This section presents analysis of the Lift \& Learn method in the setting where the data come from a consistent finite difference discretization of the original nonlinear PDE with no noise. \Cref{sec: FD setting} presents the setting of our analysis and \Cref{sec: FD error} proves an upper bound on the residual of the Lift \& Learn minimization.

\subsection{Consistent finite difference models}\label{sec: FD setting}
We now provide a local error analysis for the setting where \cref{eq: full model} arises from a consistent finite-difference discretization of the state. Let $\bar\s$ denote the vector of original state values at the grid points $\{x_l\}_{l=1}^n$:
\begin{align}
	\bar\s(t) = \begin{pmatrix}
		s_1(x_1,t) \\ \vdots \\  s_1(x_n,t) \\ 
		\vdots \\
		s_{d_s}(x_1,t) \\ \vdots \\ s_{d_s}(x_n,t)
	\end{pmatrix}.
\end{align}
Note that $\bar\s(t)$ differs from $\s(t)$ in that $\bar\s(t)$ is the exact continuous (strong) solution of the original PDE (\cref{eq: nonlinear PDE}) evaluated at the grid points, and $\s(t)$ is the semi-discrete solution of the spatially discrete set of ODEs (\cref{eq: full model}). 
If \cref{eq: full model} is a consistent order-$p$ discretization of \cref{eq: nonlinear PDE}, then for any given $n$ there exists a constant $c_s>0$ such that
\begin{align}
	\left|\f_{(j-1)n+l}(\bar \s(t)) -  f_j(\vs(x,t))\big|_{x=x_l} \right| \leq c_s n^{-p}, \quad \forall t,
	\label{eq: f consistency}
\end{align}
for $j=1,2,\ldots,d_s$ and $l=1,2,\ldots,n$, where $ f_j$ is defined as in \cref{eq: vec f def} and $\f_{(j-1)n+l}$ denotes the $((j-1)n+l)$-th entry of the vector-valued nonlinear function $\f$, which corresponds to the time derivative of $s_j(x_l,t)$.

Now, let $\bar\w$ denote the discretization of the exact continuous lifted state on the same spatial grid:
\begin{align}
	\bar\w(t) = \bT(\bar\s(t)) = \begin{pmatrix}
		\cT_1(\vs(x_1,t)) \\ \vdots \\ \cT_1(\vs(x_n,t)) \\ 
		\vdots \\
		\cT_{d_w}(\vs(x_1,t) )\\ \vdots \\ \cT_{d_w}(\vs(x_n,t))
	\end{pmatrix}.
	\label{eq: bar bw def}
\end{align}
Then, 
\begin{align}
	\dv{\w}t = \A\w + \H(\w\otimes\w), \quad \A\in\R^{nd_w\times nd_w},\,\H\in\R^{nd_w\times n^2d_w^2}
	\label{eq: lifted ODEs}
\end{align}
is a consistent order-$p$ discretization of \cref{eq: lifted PDE} if for any given $n$ there exists a constant $c_w>0$ such that, for all $t$,
\begin{align}
	\left | \A_{(j-1)n+l,:}\bar\w + \H_{(j-1)n+l,:}(\bar\w\otimes\bar\w) - \left(a_j(\vw) + h_j(\vw)\right)\big|_{x=x_l} \right| \leq c_wn^{-p},
	\label{eq: A H consistency}
\end{align}
for $j=1,2,\ldots,d_w$ and $l=1,2,\ldots,n$, where $\A_{(j-1)n+l,:}$ and $\H_{(j-1)n+l,:}$ denote the $((j-1)n+l)$-th rows of $\A$ and $\H$, respectively, which correspond to the time derivative of the $j$-th lifted state at the $l$-th spatial node.

We note that in Lift \& Learn, we assume that the discretized nonlinear model (\cref{eq: full model}) is available, and that it is stable. In contrast, for the lifted system, we assume only that the discrete operators $\A$ and $\H$ for a consistent discretization exist --- they are used for analysis only, and availability is not required.

\subsection{Theoretical results}\label{sec: FD error}
We now prove that the minimum achieved by the Lift \& Learn model in \cref{eq: transposed minimization} is bounded above by the objective value achieved by the intrusive lifted reduced model. This demonstrates two advantages of our method:
\begin{enumerate}
	\item our non-intrusively obtained model is able to model the data at least as well as an intrusive reduced model, and 
	\item unlike many other learning methods, the quadratic form of the model we learn respects the model physics, enabling us to put an upper bound on the residual of the Lift \& Learn model on the training data.
\end{enumerate}

We begin with a consistency result.
\begin{lemma}
	If $\bar\w(t)$ is defined as in \cref{eq: bar bw def}, and if \cref{eq: full model} and \cref{eq: lifted ODEs} are consistent discretizations of \cref{eq: nonlinear PDE} and \cref{eq: lifted PDE}, respectively, then 
	\begin{align}
		\left\|\A\bar\w + \H(\bar\w \otimes\bar\w) - \mathbf{J}(\bT^\dagger(\bar \w))\f(\bT^\dagger(\bar\w))\right\|_2 \lesssim n^{\frac12-p},\quad \forall t,
	\end{align}
	where $\lesssim$ denotes a bound up to a constant independent of $n$.
	\label{lem: discrete flow map bound}
\end{lemma}

\begin{proof}
	Because $\bT$ is defined by applying $\cT$ node-wise, the $((j-1)n+l)$-th row of $\mathbf{J}$ contains partial derivatives of the $j$-th lifted state at the $l$-th spatial node with respect to $\s$. This corresponds to the $j$-th row of $\mathcal{J}$ evaluated at $l$-th spatial node. That is, for all $t$,
	\begin{align}
		\mathbf{J}_{(j-1)n+l,:}(\bar\s)\f(\bar\s) &= \sum_{i=1}^{d_w}\mathbf{J}_{(j-1)n+l\,,\,(i-1)n+l}\f_{(i-1)n+l}(\bar \s) \\
		&= \sum_{i=1}^{d_w}\mathcal{J}_{j,i}(\vs(x_l))\f_{(i-1)n+l}(\bar \s) ,
	\end{align}
	for all $j=1,2,\ldots, d_w$, $l=1,2,\ldots,n$. Then, for all $t$,
	\begin{align}
		\bigg |\mathbf{J}_{(j-1)n+l,:}&(\bar\s)\f(\bar\s) - \mathcal{J}_j(\vs(x_l)) \vf(\vs)\big)\big|_{x=x_l} \bigg | \nonumber\\
		&= \left|\sum_{i=1}^{d_w} \left(\mathcal{J}_{j,i}(\vs(x_l)) \f_{(i-1)n+l}(\bar\s) - \mathcal{J}_{j,i}(\vs(x_l))\vf_i(\vs)|_{x=x_l}\right) \right| \\
		&\leq \sum_{i=1}^{d_w} \big |\mathcal{J}_{j,i}(\vs(x_l)) \big | \cdot \big| \f_{(i-1)n+l}(\bar\s) - \vf_i(\vs)|_{x=x_l} \big|.
	\end{align}
	Since $\cT$ is has bounded derivative, $\mathcal{J}$ is bounded. Then, because $\f$ is consistent, we have, for all $t$,
	\begin{align}
		\bigg |\mathbf{J}_{(j-1)n+l,:}&(\bar\s)\f(\bar\s) - \mathcal{J}_j(\vs(x_l)) \big(\vf(\vs)\big)\big|_{x=x_l} \bigg | \lesssim n^{-p}.
		\label{eq: discrete F diff cts f}
	\end{align}
	By definition, the lifted dynamics are exact for $\vw = \cT(\vs)$, so for all $j,k$, and $t$,
	\begin{align}
		\mathcal{J}_j(\vs(x_l))\big(\vf(\vs)\big)\big|_{x=x_l} = \left.\left(a_j(\vw) + h_j(\vw)\right)\right|_{x=x_l}.
		\label{eq: lifting dynamics equal}
	\end{align}
	Since $\bar\s = \bT^\dagger(\bar\w)$, we can use the triangle inequality to combine \cref{eq: A H consistency,eq: discrete F diff cts f,eq: lifting dynamics equal} as follows:
	\begin{align}
		\left |\A_{(j-1)n+l,:}\bar\w + \H_{(j-1)n+l,:}(\bar\w\otimes\bar\w) - \mathbf{J}_{(j-1)n+l,:}(\bT^\dagger(\bar\w))\f(\bT^\dagger(\bar\w)) \right | \lesssim n^{-p}
	\end{align}
	for all $j = 1,2,\ldots, d_w$, $l = 1,2,\ldots,n$, and $t$. Since there are $nd_w$ discrete lifted states, we have for all $t$ the following bound (up to a constant):
	\begin{align}
		\left\|\A\bar\w + \H(\bar\w \otimes\bar\w) - \mathbf{J}(\bT^\dagger(\bar\w))\f(\bT^\dagger(\bar\w))\right\|_2 \lesssim n^{-(p-\frac12)}.
	\end{align}
\end{proof}

\Cref{lem: discrete flow map bound} holds for any state $\bar\w = \bT(\bar\s)$ for which the quadratic lifted dynamics of the corresponding $\vw$ are exactly equivalent to the non-lifted dynamics (\cref{eq: lifting dynamics equal}). Let $\{\bar\w_k \}_{k=1}^K$ be a collection of snapshots of these exact lifted states and suppose that $\Ur$ is an $r$-dimensional POD basis for the $\bar\w_k$.  We now prove an analogous bound for the projected snapshots $\Ur\Ur^\top\bar\w_k$:

\begin{lemma}
	Assume $\mathbf{J}$ is also Lipschitz and define $\Pr = \Ur\Ur^\top$. Then, there exist constants $C_0$, $C_1$, and $C_2$ such that for all $k$
	\begin{align}
		\big\|\A\Pr\bar\w_k + \H(\Pr\bar\w_k \otimes\Pr\bar\w_k&) - \mathbf{J}(\bT^\dagger(\Pr\bar\w_k))\f(\bT^\dagger(\Pr\bar\w_k))\big\|_2 \nonumber\\
		& \leq C_0n^{\frac12-p} + (C_1 + C_2)\|(\Pr - \mathbf{I}_{nd_w})\bar\w_k\|,
		\label{eq: projected state error}
	\end{align}
	where $\mathbf{I}_{nd_w}$ is the identity matrix of dimension $nd_w$.
	\label{lem: projected state lemma}
\end{lemma}

\begin{proof}
	We begin by breaking the left side of \cref{eq: projected state error} into the following terms using the triangle inequality
	\begin{equation}
		\begin{aligned}
			\big\|\A\Pr\bar\w_k + &\H(\Pr\bar\w_k \otimes\Pr\bar\w_k) - \mathbf{J}(\bT^\dagger(\Pr\bar\w_k))\f(\bT^\dagger(\Pr\bar\w_k))\big\| \\
			&\leq\left\|\A\Pr\bar\w_k + \H(\Pr\bar\w_k \otimes\Pr\bar\w_k) - \big(\A\bar\w_k + \H(\bar\w_k \otimes\bar\w_k)\big)\right\| \\
			&\qquad+ \left\|\A\bar\w_k + \H(\bar\w_k \otimes\bar\w_k) - \mathbf{J}(\bT^\dagger(\bar\w_k))\f(\bT^\dagger(\bar\w_k))\right \| \\
			&\qquad+ \left \|\mathbf{J}(\bT^\dagger(\bar\w_k))\f(\bT^\dagger(\bar\w_k)) - \mathbf{J}(\bT^\dagger(\Pr\bar\w_k))\f(\bT^\dagger(\Pr\bar\w_k))\right\|.
		\end{aligned}
		\label{eq: proj state err triangle}
	\end{equation}
	The middle term can be bounded by \Cref{lem: discrete flow map bound}. For the first term, define $\mathbf{Q}(\w) = \A\w + \H(\w\otimes\w)$. Note that since $\mathbf{Q}$ is polynomial, it is Lipschitz over the finite set of snapshots, so there exists a constant $C_1$ such that
	\begin{align}
		\big\|\A\Pr\bar\w_k + \H(\Pr\bar\w_k \otimes\Pr\bar\w_k) - \big(\A\bar\w_k + &\H(\bar\w_k \otimes\bar\w_k)\big) \big\|  \leq C_1\|(\Pr - \mathbf{I}_{nd_w})\bar\w_k\|.
		\label{eq: first term}
	\end{align}
	For the third term on the right side of \cref{eq: proj state err triangle}, since $\mathbf{J}$ and $\f$ are Lipschitz, the function $\mathbf{J}(\bT^\dagger(\cdot))\f(\bT^\dagger(\cdot))$ is also Lipschitz over the finite set of snapshots. Then, for some constant $C_2$,
	\begin{align}
		\big \|\mathbf{J}(\bT^\dagger(\bar\w_k))\f(\bT^\dagger(\bar\w_k)) - \mathbf{J}(\bT^\dagger(\Pr\bar\w_k))&\f(\bT^\dagger(\Pr\bar\w_k))\big\| \leq C_2\|(\Pr - \mathbf{I}_{nd_w})\bar\w_k\|.
		\label{eq: third term}
	\end{align}
	Combining \cref{eq: first term,eq: third term} with \Cref{lem: discrete flow map bound} yields the result.
\end{proof}

This result lets us upper-bound the residual of the Lift \& Learn minimization (\cref{eq: transposed minimization}).

\begin{theorem}
	Let $\sigma_i$ be the singular values of the snapshot data matrix $\W$. Then, the residual of the Lift \& Learn operator inference is bounded as follows:
	\begin{align}
		\min_{\hat\A\in\R^{r\times r},\hat\H\in\R^{r\times r^2}} \frac1K \bigg\| \hat\W^{\top}\hat\A^\top + \left(\hat\W\otimes\hat\W\right)^\top&\hat\H^\top - {\hat\W}'^\top \bigg\|_F^2\nonumber \\
		& \leq \big(C_0n^{\frac12-p} + (C_1+C_2)\varepsilon\big)^2,
		\label{eq: LL opinf bound}
	\end{align}
	where $\varepsilon^2=\sum_{i=r+1}^{nd_w}\sigma_i^2$ and $C_0,C_1,C_2\geq0$ are constants.
	\label{lem: least square bound}
\end{theorem}

\begin{proof}
	Apply \Cref{lem: projected state lemma} to each projected state snapshot $\Pr \bar\w_k$:
	\begin{align}
		\left \| \A \Pr\bar\w_k + \H(\Pr\bar\w_k \otimes \Pr\bar\w_k) - \mathbf{J}(\bT^\dagger(\Pr\bar\w_k))\f(\bT^\dagger(\Pr\bar\w_k))
		\right\|_2\nonumber\\
		 < C_0n^{\frac12-p} + (C_1+C_2)\|(\Pr-\mathbf{I}_{nd_w})\bar \w_k\|.
	\end{align}
	Then, for each snapshot, 
	\begin{align}
		\left \|\Ur^\top \left(\A \Pr\bar\w_k + \H(\Pr\bar\w_k \otimes \Pr\bar\w_k) - \mathbf{J}(\bT^\dagger(\Pr\bar\w_k))\f(\bT^\dagger(\Pr\bar\w_k)) \right)
		\right\|_2 \nonumber\\
		 \leq C_0n^{\frac12-p} + (C_1+C_2)\|(\Pr-\mathbf{I}_{nd_w})\bar \w_k\|
		\label{eq: single column LS res bound}
	\end{align}
	since $\|\Ur\|_2=1$ because its columns are orthonormal vectors.
	Since the snapshots in the Lift \& Learn data are used to compute the POD basis, 
	\begin{align}
		\|(\Pr-\mathbf{I}_{nd_w})\bar\w_k\| \leq \varepsilon,
	\end{align}
	for all snapshots $\bar\w_k$.
	Thus, taking the mean of the square over all snapshots, 
	\begin{align}
		\frac1K\sum_{k=1}^K \big \|& \Ur^\top \big(\A \Ur\Ur^\top\bar\w_k + \H( \Ur\Ur^\top\bar\w_k \otimes  \Ur\Ur^\top\bar\w_k) \nonumber\\
		& - \mathbf{J}(\bT^\dagger(\Pr\bar\w_k))\f(\bT^\dagger(\Pr\bar\w_k)) \big)
		\big\|_2^2 \leq  \big(C_0n^{\frac12-p} + (C_1+C_2)\varepsilon\big)^2.
		\label{eq: final LS bound}
	\end{align}
	Note that this sum is exactly the objective function in \cref{eq: transposed minimization}. Thus, choosing $\hat\A = \Ur^\top \A \Ur$ and $\hat\H = \Ur^\top\H(\Ur\otimes\Ur)$ in \cref{eq: transposed minimization} would yield an objective value less than $\big(C_0n^{\frac12-p} + (C_1+C_2)\varepsilon\big)^2$, so the minimizer must be bounded by this value as well.
\end{proof}

The least-squares residual is a measure of the mismatch between the postulated quadratic model form and the true dynamics of the system. \Cref{lem: least square bound} shows that this mismatch has an upper bound dependent on two factors: the truncation error of the full model and the projection error of the POD basis.


\section{Results}\label{sec: results}
In this section, the Lift \& Learn method is applied to two different dynamical systems. \Cref{sec: FHN} considers the FitzHugh-Nagumo prototypical neuron activation model and \Cref{sec: Euler} considers the Euler equations for inviscid fluid flow.

\subsection{Application to the FitzHugh-Nagumo system}\label{sec: FHN}
The FitzHugh-Nagumo system was first proposed in 1961~\cite{fitzhugh1961impulses} and then realized as a circuit for electronic pulse transmission in~\cite{nagumo1962active}. The system has been used to study exciteable oscillatory dynamics such as those found in cardiac~\cite{rocsoreanu2012fitzhugh} and neuron modeling~\cite{longtin1993stochastic}, and has become a benchmark problem in nonlinear model reduction~\cite{deim2010}.
\Cref{sec: FHN statement and lifting} introduces the FitzHugh-Nagumo equations and lifts the system to quadratic form. \Cref{sec: numerical experiments} describes the data used for learning and the performance of the model on training and test sets.

\subsubsection{FitzHugh-Nagumo problem statement and lifting transformation}\label{sec: FHN statement and lifting}
The FitzHugh-Nagumo equations are a simplified neuron activation model with two states: $s_1$ represents voltage and $s_2$ represents voltage recovery. We consider the equations using the same parameters as in~\cite{deim2010}:
\begin{subequations}
  \begin{align}
    \gamma\pdv {s_1}t &= \gamma^2 \pdv[2]{s_1} x - s_1^3 + 1.1s_1^2 - 0.1s_1 + s_2 + 0.05,\label{eq: FHN v evolution}\\
    \pdv {s_2}t &= 0.5s_1 - 2s_2 + 0.05, 
  \end{align}
  \label{eq: original FHN equations}
\end{subequations}
where $\gamma=0.015$. The full model solves the system \cref{eq: original FHN equations} on the spatial domain $x\in[0,1]$ for the timeframe $t\in[0,4]$. At $t=0$, the states $s_1$ and $s_2$ are both zero everywhere, and the boundary conditions are given by
\begin{align}
    \pdv {s_1}x\bigg|_{x = 0} = g(t), \qquad \pdv {s_1}x\bigg|_{x = 1} = 0,
\end{align}
where $g(t)$ is a time-dependent input which represents a neuron stimulus.
Because \cref{eq: FHN v evolution} contains a cubic nonlinear term, the model is lifted to quadratic form. Although the operator inference framework developed in~\cite{Peherstorfer16DataDriven} can, in principle, learn cubic and higher-order tensor terms, the number of unknowns in the learning problem can increase exponentially as the polynomial order of the model is increased, if the higher-order operators are dense. 

The following lifting map $\cT$ is used in~\cite{benner2015two,kramer2018lifting} to lift the system to quadratic form:
\begin{align}
	\cT: \begin{pmatrix}
		s_1 \\ s_2 
	\end{pmatrix} \longmapsto
	\begin{pmatrix}
		s_1 \\ s_2 \\ (s_1)^2 
	\end{pmatrix} \equiv
	\begin{pmatrix}
		w_1 \\ w_2 \\ w_3
	\end{pmatrix}.
\end{align}
The lifted system is then given by
\begin{subequations}
  \begin{align}
    \gamma\pdv {w_1}t &= \gamma^2 \pdv[2]{w_1} x - w_1w_3 + 1.1(w_1)^2 - 0.1{w_1} + w_2 + 0.05,\\
      \pdv {w_2}t &= 0.5{w_1} - 2w_2 + 0.05, \\
      \pdv {w_3}t &= 2w_1\pdv {w_1}t\nonumber \\
      &= \frac2\gamma\left(\gamma^2 {w_1}\pdv[2]{w_1}x - w_3^2 + 1.1w_1w_3 - 0.1 w_3 +w_1w_2 +  0.05w_1\right).
  \end{align}
  \label{eq: lifted FHN system}
\end{subequations}
To be consistent with the prescribed initial and boundary conditions for the original state, $w_3$ must be zero everywhere at $t=0$ and its boundary conditions are given by
\begin{align}
  \pdv {w_3}x\bigg|_{x = 0} = 2{w_1}\pdv {w_1}x\bigg|_{x=0} = 2w_1g(t), \qquad \pdv {w_3}x\bigg|_{x = 1} =2{w_1}\pdv {w_1}x\bigg|_{x=L}= 0.
  \label{eq: z boundary condition}
\end{align}
The lifted system in \cref{eq: lifted FHN system} contains only quadratic nonlinear dependencies on the state. Note that no approximation has been introduced in lifting \cref{eq: original FHN equations} to \cref{eq: lifted FHN system}.

We now postulate the form of the lifted reduced model based on the lifted PDE (\cref{eq: lifted FHN system}). In addition to linear and quadratic terms, \cref{eq: lifted FHN system} also contains constant, input, and bilinear terms, so the postulated model form is given by 
\begin{align}
    \dv{\hat\w} t = \hat\A \hat\w + \hat\H (\hat\w \otimes\hat\w) + \hat{\mathbf{N}}\hat\w g(t) + \hat\B g(t) + \hat{\mathbf{C}},
    \label{eq: FHN model form}
\end{align}
where $\hat\A,\hat{\mathbf{N}}\in\R^{r\times r}$, $\hat\H\in\R^{r\times r^2}$, and $\hat{\mathbf{B}},\hat{\mathbf{C}}\in\R^r$.

\subsubsection{Numerical experiments}\label{sec: numerical experiments}
We wish to model the response of the FitzHugh-Nagumo system to inputs of the form
\begin{align}
  g(t) = \alpha t^3 e^{-\beta t}
\end{align}
where the parameter $\alpha$ varies log-uniformly between 500 and 50000, and the parameter $\beta$ varies uniformly on the range $[10,15]$. 
To train our Lift \& Learn model of the form \cref{eq: FHN model form}, snapshot data from nine simulations of the original equations (\cref{eq: original FHN equations}) corresponding to the parameters $\alpha = [500, 5000, 50000]$ and $\beta = [10, 12.5, 15]$ are generated. The simulation outputs the system state $s_1$ and $s_2$ on a uniform spatial grid with $n=512$ nodes. The data are recorded every 0.01 seconds, yielding 400 state snapshots for each simulation, with a total of 3600 snapshots used for learning. The lifting map is applied to the state data to obtain lifted data for $w_1$, $w_2$, and $w_3$. Separate POD bases are computed for each lifted state variable. That is, if $\W^{(1)},\W^{(2)},\W^{(3)}\in\R^{512\times 3600}$ denote the snapshot data in $w_1$, $w_2$, and $w_3$, respectively, then
\begin{align}
   \W^{(1)} = \U^{(1)}\mathbf{\Sigma}^{(1)}\mathbf{V}^{(1)^\top}, \quad
   \W^{(2)} = \U^{(2)}\mathbf{\Sigma}^{(2)}\mathbf{V}^{(2)^\top}, \quad
   \W^{(3)} = \U^{(3)}\mathbf{\Sigma}^{(3)}\mathbf{V}^{(3)^\top}.
\end{align} 
The number of modes to retain in our low-dimensional model is determined by examining the quantity
\begin{align}
 	1 - \sum_{i=r+1}^{nd_w}\sigma_i^2 \bigg/\sum_{i=1}^{nd_w}\sigma_i^2,
 	\label{eq: energy svd}
\end{align}
where $\sigma_i$ are the singular values of the data. The quantity in~\cref{eq: energy svd} is the relative energy of the unretained modes of the training data, or the energy lost in truncation. The energy spectrum of the training data in each separate lifted state variable is shown in \Cref{fig: FHN energy}. The required numbers of modes needed to capture 99.9\%, 99.99\%, 99.999\%, and 99.9999\% of the energy in the data are tabulated in \Cref{tab: FHN sizes}, along with the corresponding Lift \& Learn model dimensions.

\begin{figure}[h]
	\centering
	\includegraphics[width=0.8\textwidth]{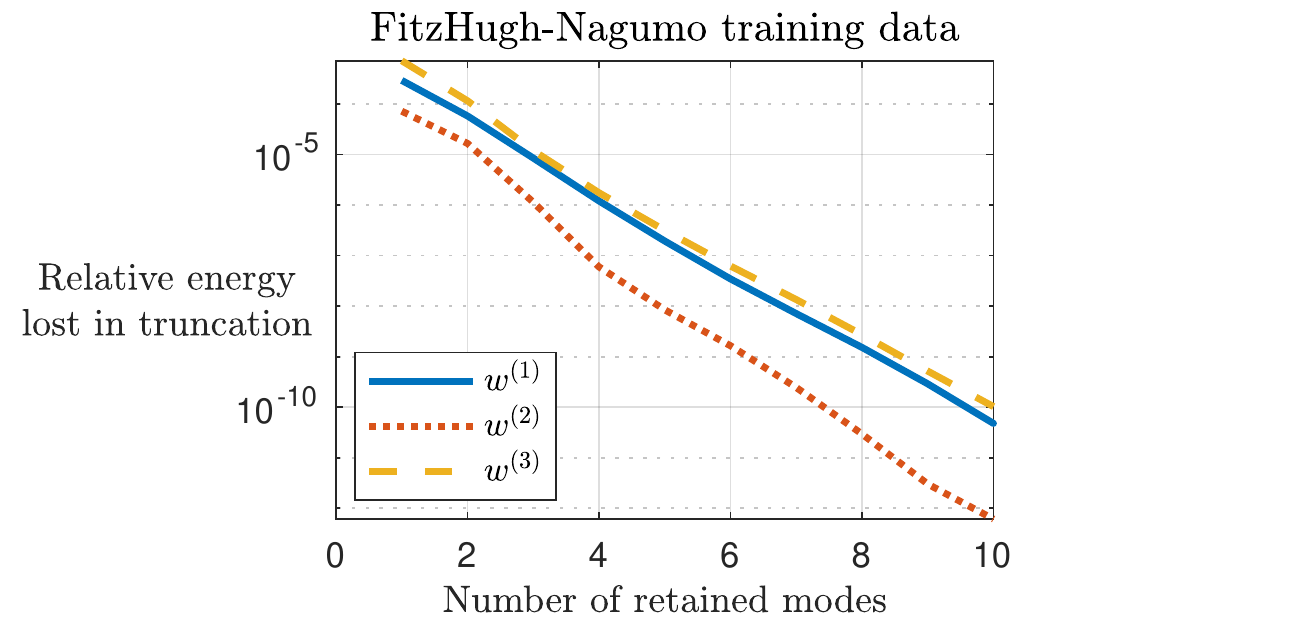}
  	\caption{Energy spectrum of FitzHugh-Nagumo training data. }\label{fig: FHN energy}
\end{figure}

\begin{table}[h]
	\centering
	\begin{tabular}{l c c c c}
    &  \multicolumn{4}{c}{\shortstack{\# modes  required}}  \\
    Retained energy & $w_1$ & $w_2$ & $w_3$ & Total\\ 
    \hline
    99.9\phantom{9999}\% & 1 & 1 & 1 & 3 \\
    99.99\phantom{999}\% & 2 & 1 & 3 & 6 \\
    99.999\phantom{99}\% & 3 & 3 & 4 & 10 \\
    99.9999\phantom{9}\% & 5 & 4 & 5 & 14
  \end{tabular}
  \caption{Number of modes required to retain different amounts of energy in training data. The total number of modes corresponds to the Lift \& Learn model dimension.}\label{tab: FHN sizes}
\end{table}

For each of the model sizes in \Cref{tab: FHN sizes}, the data are generated by evaluating the original non-quadratic model and lifting the data, as described in \Cref{ssec: data acquisition}. The corresponding input data $\mathbf{G}$ and bilinear state-input data $\hat\W\mathbf{G}$ are also collected. The state, time derivative, input, and bilinear state-input data
are then used to learn a model of the form \cref{eq: FHN model form}. Additionally, we exploit knowledge of the governing equations to enforce block-sparsity in the least-squares solution; for example, there are only linear and constant terms in the evolution of $w_2$, so for the reduced state components corresponding to $w_2$, only the linear and constant terms are inferred.

The training set consists of the nine training trajectories described above. Two test sets are considered: one in which 100 trajectories are generated with $\alpha$ and $\beta$ realizations randomly drawn from their distributions above, in the same regime as the training set, and a second test set in which 100 trajectories are generated from a different parameter regime, with $\alpha$ varying log-uniformly on $[50000,5e6]$ and $\beta\sim\mathcal{U}([15, 20])$. Training and test errors are shown in \Cref{fig: FHN training error}. For each training and test input, the error relative to the solution of the original full model, $\S_\text{orig}$, is calculated by solving the reduced model to generate predictions, then reconstructing the full lifted state from the Lift \& Learn reduced trajectory, and finally reversing the lifting to obtain $\S_\text{L\&L}$, the Lift \& Learn prediction of the trajectory in the original full state space. The relative error is given by
\begin{align}
   \frac{\|\S_\text{L\&L}-\S_\text{orig}\|_F}{\|\S_\text{orig}\|_F}.
\end{align} 
The relative error of the intrusive lifted POD reduced model (obtained by discretizing the lifted PDE and then reducing, as in~\cite{kramer2018lifting}) is shown for reference.

\begin{figure}[h]
	\centering
	\includegraphics[width = \textwidth]{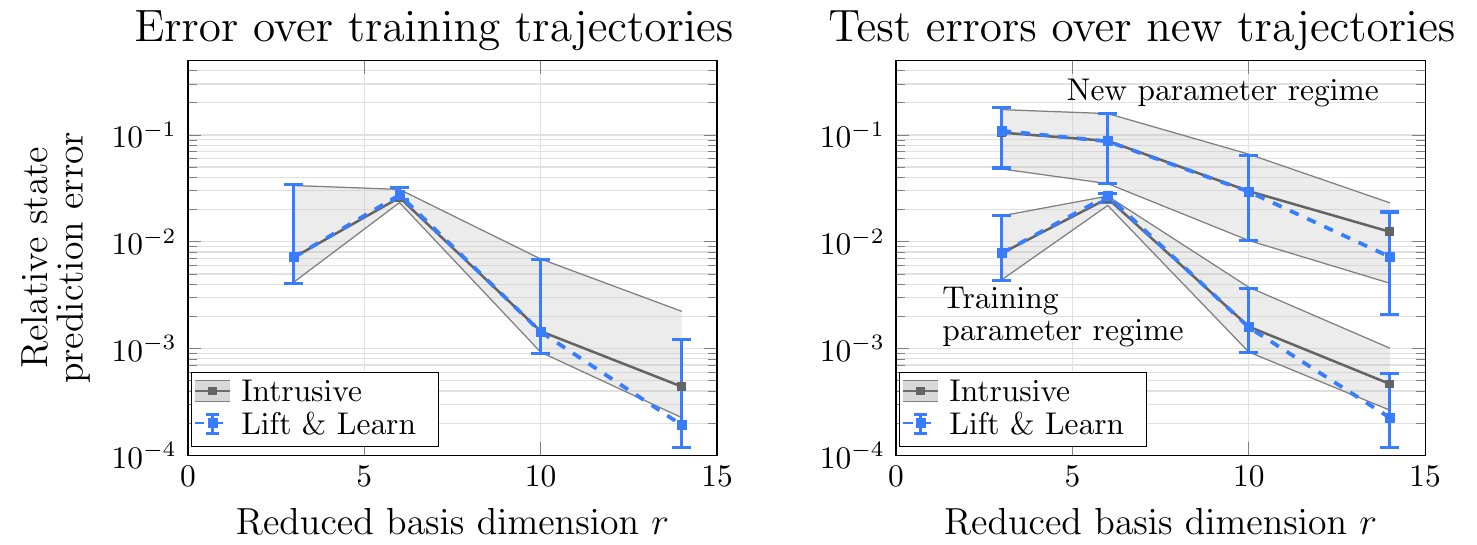}
	\caption{Lift \& Learn model prediction error for FitzHugh-Nagumo system. Comparison to intrusive lifted POD reduced model performance for (left) the nine training trajectories and (right) two test sets of 100 new trajectories: one set drawn from the same parameter regime as the training trajectories and the other set drawn from a completely new parameter regime. Median and first/third quartile errors are shown.}
	\label{fig: FHN training error}
\end{figure}


For both test sets and the training set, the Lift \& Learn model error is nearly the same as the lifted POD reduced model error for $r\leq10$. The test errors for the first test set are of similar magnitude to the training errors, demonstrating the ability of the model to generalize to new inputs. For the second test set, the Lift \& Learn model errors are higher than for the training set, but similar to the errors obtained by the intrusive lifted POD reduced model. In this case, the accuracy of the reduced model is limited by the ability of the POD basis computed from trajectories in one parameter regime to represent trajectories in a new parameter regime. The ability of the Lift \& Learn model to recover the accuracy of the lifted POD reduced model is a key contribution of this work because it allows analyzable quadratic reduced models to be derived for nonlinear systems where the lifted full model is not available.

\subsection{Application to the Euler equations}\label{sec: Euler}
The fluid dynamics community has long recognized the utility of alternative variable representations. While the Euler and Navier-Stokes equations are most commonly derived in conservative variables, where each state is a conserved quantity (mass, momentum, energy), symmetric variables have been exploited to guarantee stable models~\cite{HUGHES1986variableTransformationNS}, and the quadratic structure of the specific volume variable representation has been exploited in~\cite{balajewicz2016minimal} to allow a model stabilization procedure to be applied. In \Cref{sec:liftEuler} the Euler equations are presented in the typical conservative formulation as well as in the quadratic specific volume formulation. \Cref{sec: euler num exp} describes the test problem and presents Lift \& Learn results when applied to this problem.

\subsubsection{Euler equations and specific volume transformation}\label{sec:liftEuler}
The one-dimensional Euler equations in the conservative variables are given by
\begin{align}
	\frac{\partial}{\partial t} \begin{pmatrix}
		\rho \\ \rho u\\ \rho e
	\end{pmatrix} 
	=-\frac{\partial}{\partial x}\begin{pmatrix}
		\rho u \\ \rho u^2 + p \\ (\rho e+p)u
	\end{pmatrix},
	\label{eq: Euler equations in 1 dimension conservative}
\end{align}
where the state $\vs = \begin{pmatrix}
	\rho & \rho u & \rho e
\end{pmatrix}^\top$ is comprised of the density $\rho$, specific momentum $\rho u$, and total energy $\rho e$. The equation of state $\rho e = \frac{p}{\gamma-1} + \frac12\rho u^2$ relates energy and pressure via the heat capacity ratio $\gamma$. \Cref{eq: Euler equations in 1 dimension conservative} contains non-polynomial nonlinearities in the conservative state. However, the Euler equations can be alternatively formulated in the specific volume state representation:
\begin{subequations}
	\begin{align}
		\pdv u t& = -u \pdv ux -\zeta \pdv px, \\
		\pdv pt & = -\gamma p \pdv ux -u\pdv px,\\
		\pdv \zeta t & = -u\pdv \zeta x + \zeta\pdv ux,
	\end{align}
	\label{eq: lifted primitive equations}
\end{subequations}
where $\zeta = \frac1\rho$ is the specific volume, $u$ the velocity, and $p$ pressure. That is, the map $\cT$ is given by
\begin{align}
	\cT: \begin{pmatrix}
		\rho \\ \rho u \\ \rho e
	\end{pmatrix}
	\longmapsto
	\begin{pmatrix}
		u \\ p \\ 1/\rho
	\end{pmatrix}\equiv \vw.
\end{align}
For constant $\gamma$, \cref{eq: lifted primitive equations} contains only quadratic nonlinear dependencies on the state and its spatial derivatives. This is a special case where a quadratic representation can be achieved via nonlinear state transformations without the addition of auxiliary variables, so $d_s = d_w$ and the map $\cT$ is invertible.
Note that \cref{eq: lifted primitive equations} can be extended to the three-dimensional Euler setting by adding the $y$- and $z$-velocity equations.

\subsubsection{Numerical experiments}\label{sec: euler num exp}
\Cref{eq: Euler equations in 1 dimension conservative} is solved on the periodic domain $x\in[0,2)$ with mesh size $\Delta x = 0.01$ from $t=0$ to $t=0.01$ with timestep $\Delta t = 10^{-5}$. The initial pressure is 1 bar everywhere. The initial density is a periodic cubic spline interpolation at the points $x =0, \frac23, \frac43$, where the interpolation values are drawn from a uniform distribution on the interval $[20, 24]$ kg/$\text{m}^3$. The initial velocity is also a periodic cubic spline interpolation of values at the same $x$-locations with interpolation values drawn from a uniform distribution on the interval [95, 105] m/s. The initial condition is therefore parametrized by six degrees of freedom: three for density and three for velocity.

A training data set of 64 trajectories is constructed using initial conditions corresponding to the 64 corners of the parameter domain. This data set is then transformed to the specific volume representation and non-dimensionalized. For this test problem, we use a single POD basis to represent the entire state. The energy spectrum of the non-dimensional training data is plotted in \Cref{fig: Euler energy}. The numbers of modes required to retain different levels of energy are tabulated in \Cref{tab: Euler sizes}.

\begin{figure}[h]
  \centering
  \begin{minipage}{0.55\textwidth}
  \centering
  \includegraphics[width=0.9\textwidth]{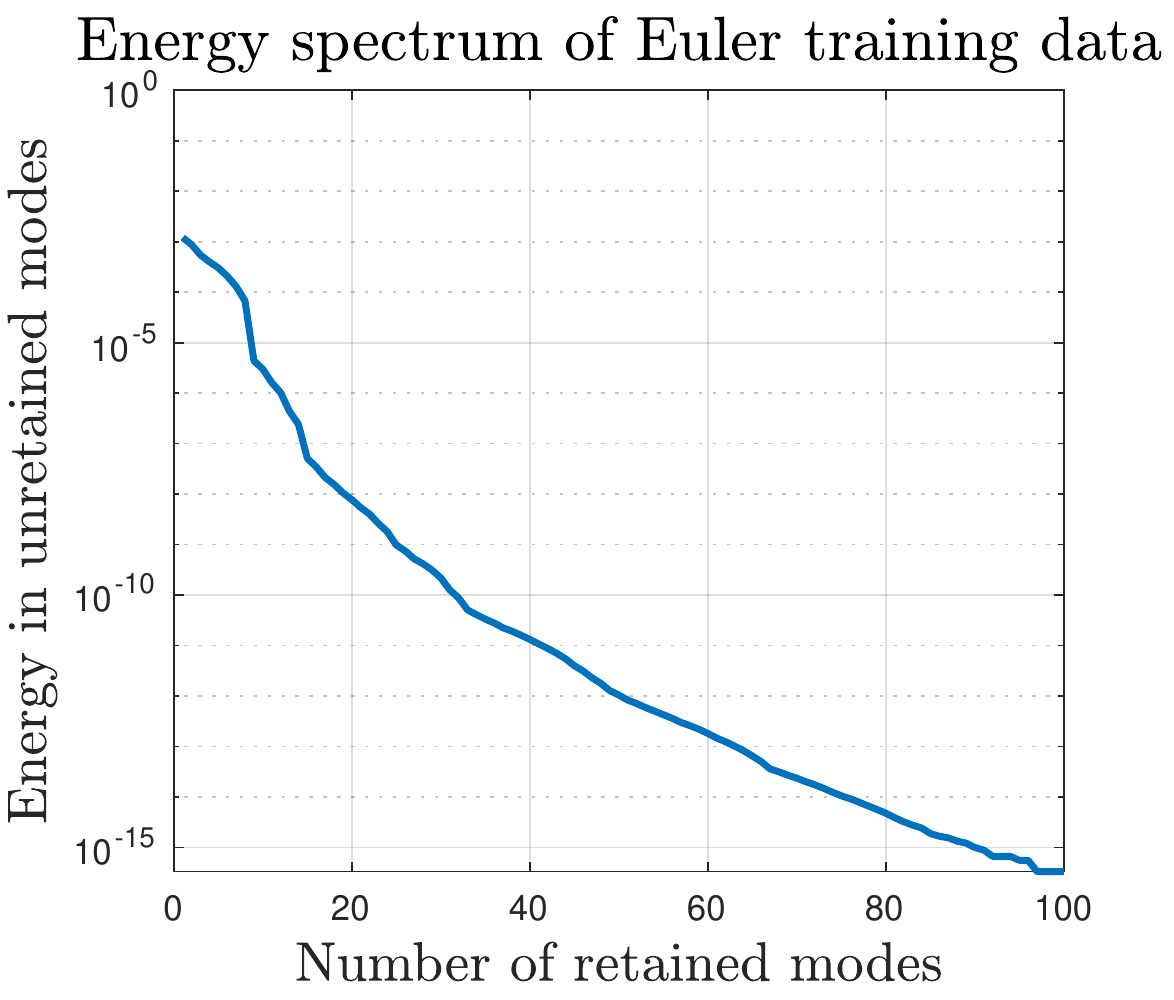}
	\caption{POD energy spectrum of Euler training data set (all state variables)}\label{fig: Euler energy}
  \end{minipage}
  \hfill
  \begin{minipage}{0.43\textwidth}
  \centering
  \captionsetup{type=table} 
  \begin{tabular}{c c }
    Retained energy & \# modes required\\ 
    \hline
    $1-10^{-3}$& 2 \\
    $1-10^{-5}$& 9\\
    $1-10^{-7}$ & 15 \\
    $1-10^{-9}$ & 25
  \end{tabular}
  \caption{Number of modes required to retain energy of training data. The total number of modes corresponds to the Lift \& Learn model dimension.}\label{tab: Euler sizes}
  \end{minipage}
\end{figure}

Note that the transformed system in \cref{eq: lifted primitive equations} contains no linear dependencies on the state, so only a quadratic operator $\hat\H$ is inferred. The inferred reduced model thus has the form
\begin{align}
	\dv{\hat\w} t = \hat\H(\hat\w \otimes \hat\w).
\end{align}
The Lift \& Learn reduced model is then used to predict the state evolution for the training data set as well as a test data set. The test set consists of 100 trajectories based on initial conditions drawn randomly from their distributions. Median and first and third quartile relative errors over the training and test sets are shown in \Cref{fig: euler state error}. The performance of a lifted POD reduced model is shown for reference. 

\begin{figure}[h]
	\centering
	\includegraphics[width=\textwidth]{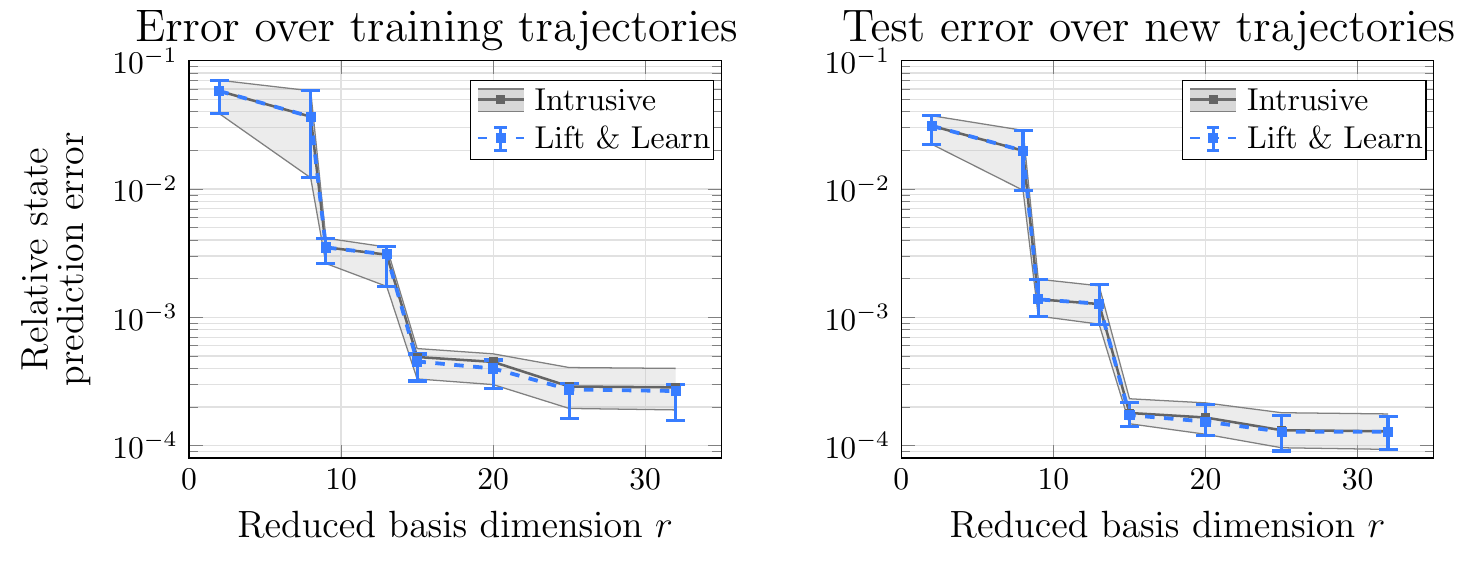}
	\caption{Lift \& Learn model prediction error for Euler equations and comparison to intrusive lifted POD reduced model performance. Errors over the 64 training trajectories are plotted on the left and test error over 100 new test trajectories are plotted on the right. Median, first and third quartile errors are shown. }
	\label{fig: euler state error}
\end{figure}

The Lift \& Learn models are stable, accurate, and generalizable, achieving an error under 0.1\% on both the training and test sets, which is similar to that of the lifted POD reduced model. 
For this particular training and test set, both the intrusive and the Lift \& Learn models actually achieve a lower test error than training error. This may be because the state data of the test set generated from random parameters uniformly distributed across their domains is more accurately represented by the low-dimensional POD basis than the state data of the training set generated by parameters at the domain extremes. We emphasize that it is our analysis bounding difference between the residuals of the Lift \& Learn model and intrusive reduced models that allows us to interpret the learned model behavior through the lens of projection-based reduced models in this way.
Again, the lifted POD reduced model approach is usually not viable because the lifted full model is generally not available. The non-intrusive ability to recover the generalizability of the intrusive reduced model is a key contribution of our method. 

\section{Conclusions}
For machine learning tools to achieve their full potential in scientific computing, they must be reliable, robust, and generalizable. Approaches which incorporate domain knowledge into the learning problem can help achieve these goals.
We have presented Lift \& Learn, a new domain-aware learning method which uses lifting transformations to expose quadratic structure in a nonlinear dynamical system. This structure is then exploited to learn low-dimensional models which generalize well to conditions outside of their training data. 
Our method differs from previous works that use lifting to obtain reduced models in that our approach learns the quadratic reduced model from data generated by a model in the original nonlinear variables; we do not require the availability of a lifted full model. In contrast to learning approaches that fit models with arbitrary architectures to data, our physics-informed approach fits a quadratic model that respects the physics in the lifted variables, so that the learned model residual can be upper bounded by the truncation error of the full model and the projection error of the reduced basis.
Numerical experiments on two different test problems demonstrate the ability of our learned models to recover the robustness and accuracy of the intrusive reduced models. 
However, several open questions remain and are intriguing subjects for future work, including theoretical questions of the existence of quadratic liftings, the minimal dimension of the lifted state, the automated discovery of lifted coordinates, the use of our framework to learn models from noisy or experimental data in addition to or instead of simulation data, and the incorporation of additional physical knowledge such as symmetries or stability to the learning problem via adding constraints to the optimization.


\section*{Acknowledgments}
This work was supported in part by the US Air Force Center of Excellence on Multi-Fidelity Modeling of Rocket Combustor Dynamics award FA9550-17-1-0195, the Air Force Office of Scientific Research MURI on managing multiple information sources of multi-physics systems awards FA9550-15-1-0038 and FA9550-18-1-0023, the US Department of Energy Applied Mathematics MMICC Program award DESC0019334, and the SUTD-MIT International Design Centre. The first author also acknowledges support from the National Science Foundation Graduate Research Fellowship Program and the Fannie and John Hertz Foundation.
The third author was partially supported by the US Department of Energy, Office of Advanced Scientific Computing Research, Applied Mathematics Program (Program Manager Dr. Steven Lee), DOE Award DESC0019334.

\appendix




\bibliographystyle{model1-num-names}
\bibliography{sample.bib}







\end{document}